\documentclass[10pt,twoside,a4paper]{article}

\usepackage[english]{babel}
\usepackage{amsmath}
\usepackage{amsthm}
\usepackage{amsfonts}
\usepackage{latexsym}
\usepackage{amsxtra}
\usepackage{graphicx}

\newtheorem{theorem}{Theorem}[section]
\newtheorem{corollary}{Corollary}[section]
\newtheorem{lemma}{Lemma}[section]

\newtheorem{proposition}{Proposition}[section]

\DeclareMathOperator{\ind}{\mathrm{ind}}
\DeclareMathOperator{\fix}{\mathrm{Fix}}
\DeclareMathOperator{\dist}{\mathrm{dist}}
\newcommand{\cl}[1]{\overline{#1}}
\newcommand{\R}{\mathbb{R}}
\newcommand{\N}{\mathbb{N}}
\newcommand{\lvee}{\mbox{\scalebox{1}[0.35]{$\vee$}}}
\newcommand{\h}[1]{\widehat{#1}}
\renewcommand{\t}[1]{\widetilde{#1}}
\newcommand{\ck}[1]{{\stackrel{\lvee}{#1}}}

\newcommand{\V}{\mathcal{V}}
\newcommand{\dom}[1]{\mathcal{D}(#1)}


\title{
{\bf\Large  Periodic perturbations with delay of autonomous differential equations on 
manifolds}}
\author{{\bf\large Massimo Furi}\vspace{1mm}\\
{\it\small Dipartimento di Matematica Applicata `G.\ Sansone'},\\ {\it\small Universit\`a di
Firenze}, {\it\small Via S.\ Marta 3, 50139 Firenze, Italia}\\
{\it\small e-mail: massimo.furi@unifi.it}\vspace{1mm}\\
{\bf\large Marco Spadini}\vspace{1mm}\\
{\it\small Dipartimento di Matematica Applicata `G.\ Sansone'},\\ {\it\small Universit\`a di
Firenze}, {\it\small Via S.\ Marta 3, 50139 Firenze, Italia}\\
{\it\small e-mail: marco.spadini@unifi.it}\vspace{1mm}}

\begin{document}

\maketitle

\begin{center}
{\bf\small Abstract}

\vspace{3mm}
\hspace{.05in}\parbox{4.5in}
{{\small We apply topological methods to the study of the set of harmonic 
solutions of periodically perturbed autonomous ordinary differential 
equations on differentiable manifolds, allowing the perturbing term to 
contain a fixed delay.

In the crucial step, in order to cope with the delay, we define a suitable 
(infinite dimensional) notion of Poincar\'e $T$-translation operator and 
prove a formula that, in the unperturbed case, allows the computation of its 
fixed point index.}}
\end{center}

\section{Introduction}

In this paper we shall study the set of harmonic solutions to periodic perturbations 
of autonomous ODEs on (smooth) manifolds, allowing for the perturbation to
contain a delay. Namely, given $T>0$, $r\geq 0$ and a manifold $M\subseteq\R^k$,
we will consider the $T$-periodic solutions to
\begin{equation}\label{eqintro}
 \dot x(t)=g\big(x(t)\big)+\lambda f\big(t,x(t),x(t-r)\big),\quad\lambda\geq 0,
\end{equation}
where $g$ is a tangent vector field on $M$ and $f$ is $T$-periodic in $t$ and tangent 
to $M$ in the second variable (the meaning of these terms will be explained in due 
course). Roughly speaking, we will give conditions ensuring the existence of a 
connected component of pairs $(\lambda,x)$, $\lambda\geq 0$ and $x$ a $T$-periodic 
solution to the above equation, that emanates from the set of zeros of $g$ and is not 
compact. We point out that, although this result is valid for $f$ and $g$ merely 
continuous, its proof boils down, by an approximation procedure, to the case when $f$ 
and $g$ are $C^1$.

Carrying out this program requires topological tools like the fixed point index and 
the degree (also called the rotation or characteristic) of a tangent vector field, 
that shall be briefly recalled in Section 2. In fact, in the case when the perturbation 
$f$ is independent of the delay, as in \cite{FS97}, the existence of such a connected 
component of $T$-periodic solutions is based on the computation of the fixed point index 
of the translation operator (at time $T$) associated to the equation \eqref{eqintro} 
when $f$ and $g$ are $C^1$. This computation is derived from a formula (see e.g.\ 
\cite{FS96}) that relates the degree of $g$ with the fixed point index of the finite
dimensional Poincar\'e $T$-translation operator $P$ at time $T$ associated to the 
unperturbed equation $\dot x=g(x)$. However, since in our case the perturbing term $f$ 
contains a delay, the $T$-translation operator $P$ must be replaced by its infinite 
dimensional version. Namely, the operator $Q$ that to any function 
$\varphi\in\t M:=C([-r,0],M)$ associates the function of $\t M$ given by 
$\theta\mapsto x\big(\varphi(0),\theta+T\big)$. Here $x(p,\cdot)$ denotes the unique 
solution to the Cauchy problem
\[
 \dot x=g(x),\quad x(0)=p.
\]
Clearly $P$ and $Q$ are closely related, although they operate in different spaces, only 
one of which is finite dimensional. The relation between these operators is discussed 
in Section 3, where we derive a formula that deduces the fixed point index of $Q$ from 
the degree of the tangent vector field $g$. 

Let us be more precise about the above mentioned formulas for the fixed point indices of
$P$ and $Q$. It has been proved in \cite{FS96} that, given $U\subseteq M$ open, one has
\begin{equation}\label{findexP}
 \ind(P,U)=\deg(-g,U),
\end{equation}
provided that the left hand side member of \eqref{findexP} is defined. Formula 
\eqref{findexP} is a generalization of a result of \cite{CMZ} valid for $M=\R^k$, that
 was related to an earlier theorem by Krasnosel'ski\u\i{} \cite{Kr}. This latter result 
holds for nonautonomous differential equations on manifolds, but requires a rather 
restrictive assumption called \emph{$T$-irreversibility} (which, in our settings, simply 
means that the closure $\cl{U}$ of $U$ in $M$ is compact and the map $p\mapsto x(p,t)$ is 
fixed point free on $\partial U$ for all $t\in (0,T]$).
Equation \eqref{findexP} does away with this heavy assumption and allows, by means of
the properties of Commutativity of the fixed point index and Excision of the degree, to 
deduce a similar formula for $Q$. In fact, given $W\subseteq\t M$ open, we have 
\begin{equation}\label{findexQ}
 \ind(Q,W)=\deg(-g,\ck W),
\end{equation}
provided that the left hand side member is defined. Here $\ck W$ denotes the set of points 
of $M$ that, when regarded as constant functions of $\t M$, belong to $W$.

The formula described above for the computation of the fixed point index of $Q$ allows us,
in Sections 4 and 5, to follow the lines of \cite{FS97} in order to prove our main result
about the connected sets of $T$-periodic solutions $(\lambda,x)$. We point out that in 
Sections 3 and 4 the maps $f$ and $g$ are always considered $C^1$, while in Section 5 the 
merely continuous case is considered.

For what concerns the basic theory of delay differential equations we refer to the book 
\cite{HL} and to the paper \cite{Ol69}.

\section{Preliminaries and notation}
This section is devoted to some facts and notation that will be needed in
this paper. In particular we recall the notions of fixed point index of a map and of 
degree of a tangent vector field.

Let us begin with the fixed point index. We recall that a metrizable space $E$ is an 
\emph{absolute neighborhood retract} (ANR) if, whenever it is homeomorphically embedded 
as a closed subset $C$ of a metric space $X$, there exists an open neighborhood $U$ of 
$C$ in $X$ and a retraction $r:U\to C$ (see e.g.\ \cite{Bo,DuGr}). Polyhedra and 
differentiable manifolds are examples of ANRs. Let us also recall that a continuous map 
between topological spaces is called \emph{locally compact} if it has the property that 
each point in its domain has a neighborhood whose image is contained in a compact set.

Let $E$ be an ANR and let $\psi:\dom{\psi}\to E$ be a locally compact map defined on an 
open subset $\dom{\psi}$ of $E$. Given an open subset $U$ of $\dom{\psi}$, if the set 
$\fix{(\psi,U)}$ of the fixed points of $\psi$ in $U$ is compact, then it is well defined 
an integer, $\ind(\psi,U)$, called \emph{the fixed point index of $\psi$ in $U$}
(see, e.g.\ \cite{DuGr,G,N}). Roughly speaking, $\ind(\psi,U)$ counts algebraically the 
elements of $\fix{(\psi,U)}$. 

The fixed point index turns out to be completely determined by the following four 
properties that, therefore, could be used as axioms (see \cite{Br}). Here, $E$ is an ANR 
and $U\subseteq E$ is open. 
\begin{description}
 \item[Normalization.]{\em Let $\psi:E\to E$ be constant. Then $\ind(\psi,E)=1$.}

 \item[Additivity.]{\em Given a locally compact map $\psi:U\to E$ with $\fix(\psi,U)$ 
compact, if $U_1$ and $U_2$ are disjoint open subsets of $U$ such that 
$\fix(\psi,U)\subseteq U_1\cup U_2$, then}
\[
\ind(\psi,U) = \ind(\psi,U_1)+\ind(\psi,U_2).
\]

 \item[Homotopy Invariance.]{\em Assume that $H:U\times[0,1]\to E$ is an admissible 
homotopy in $U$; that is, $H$ is locally compact and the set 
$\{(x,\lambda)\in U\times[0,1]:H(x,\lambda)=x\}$ is compact. Then}
\[
\ind \big(H(\cdot,0),U \big) = \ind \big(H(\cdot,1),U \big).
\]

\item[Commutativity.]{\em Let $E_1$, $E_2$ be ANRs and let $U_1\subseteq E_1$ and
$U_2\subseteq E_2$ be open. Suppose $\psi_1:U_1\to E_2$ and $\psi_2:U_2\to E_1$ are
locally compact maps. If one of the sets 
\[
 \{x\in \psi_1^{-1}(U_2):\psi_2\circ\psi_1(x)=x\}\quad\text{or}\quad
\{y\in \psi_2^{-1}(U_1):\psi_1\circ\psi_2(y)=y\}
\]
is compact, then so is the other and}
\[
\ind\big(\psi_2\circ\psi_1,\psi_{1}^{-1}(U_2)\big)=
\ind\big(\psi_1\circ\psi_2,\psi_{2}^{-1}(U_1)\big).
\]
\end{description}

It is easily shown that the Additivity Property implies the following two important ones:
\begin{description}
\item[Solution.] Let $\psi:U\to E$ be locally compact with $\fix(\psi,U)=\emptyset$. Then
$\ind(\psi,U)= 0$.
\item[Excision.]{\em Given a locally compact map $\psi:U\to E$ with $\fix(\psi,U)$ compact, 
and an open subset $V$ of $U$ containing $\fix(\psi,U)$, one has 
$\ind(\psi,U) = \ind(\psi,V)$.}
\end{description}

From the Homotopy Invariance and Excision properties one could deduce the following 
property:
\begin{description}
 \item[Generalized Homotopy Invariance.]{\em  Let $W\subseteq E\times[0,1]$ be open. 
Assume that $H:W\to E$ is locally compact and such that the set 
$\{(x,\lambda)\in W:H(x,\lambda)=x\}$ is compact. Let $W_\lambda$ denote the slice 
$W_\lambda:=\{x\in E:(x,\lambda)\in W\}$.  Then, 
$\ind\big(H(\cdot,\lambda),W_\lambda\big)$ does not depend on $\lambda\in[0,1]$.}
\end{description}

In the case when $E$ is a finite dimensional manifold, the fixed point index is 
uniquely determined by the first three properties (see \cite{FPS04}). It is also 
worth mentioning that when $E=\R^n$, $U$ is bounded, $\psi$ is defined on $\cl U$ and
fixed point free on $\partial U$, then $\ind(\psi,U)$ is just the Brouwer degree 
$\deg_B(I-\psi,U,0)$, where $I$ denotes the identity on $\R^n$.

\smallskip
We now recall some basic notions about tangent vector fields on manifolds.

Let $M\subseteq\R^k$ be a manifold. Given any $p\in M$, $T_pM\subseteq\R^k$ 
denotes the tangent space of M at $p$. Let $w$ be a tangent vector field on $M$, that is, 
a (continuous) map $w:M\to\R^k$ with the property that $w(p)\in T_pM$ for any $p\in M$. If
$p\in M$ is such that $w(p)=0$, then the Fr\'echet derivative $w'(p) : T_pM\to\R^k$ maps 
$T_pM$ into itself (see e.g.\ \cite{Mi}), so that the determinant $\det w'(p)$ of $w'(p)$ 
is defined. If, in addition, $p$ is a nondegenerate zero (i.e.\ $w'(p) : T_pM\to\R^k$ is 
injective) then $p$ is an isolated zero and $\det w'(p)\neq 0$.

Let $U$ be an open subset of $M$ in which we assume $w$ admissible for the degree; that is, 
the set $w^{-1}(0) \cap U$ is compact. Then, one can associate to the pair $(w,U)$ an integer, 
$\deg(w,U)$, called the \emph{degree (or characteristic) of the vector field $w$ in $U$}, 
which, roughly speaking, counts (algebraically) the zeros of $w$ in $U$ (see e.g.\ 
\cite{H, Mi, FPS05} and references therein). For instance, when the zeros of $w$ are all 
nondegenerate, then the set $w^{-1}(0)\cap U$ is finite and
\[
\deg(w,U)=\sum_{q\in w^{-1}(0)\cap U}{\rm sign} \det w'(q).
\]
When $M = \R^k$, $\deg(w,U)$ is just the classical Brouwer degree, $\deg_B(w,V,0)$, where 
$V$ is any bounded open neighborhood of $w^{-1}(0) \cap U$ whose closure is in $U$. 
Moreover, when $M$ is a compact manifold, the celebrated Poincar\'e-Hopf Theorem states 
that $\deg (w,M)$ coincides with the Euler-Poincar\'e characteristic $\chi(M)$ of $M$ 
and, therefore, is independent of $w$. 

For the pourpose of future reference, we mention a few of the properties of the degree 
of a tangent vector field that shall be useful in the sequel. Here $U$ is an open subset 
of a manifold $M\subseteq\R^k$ and $g:M\to\R^k$ is a tangent vector field.
\begin{description}
\item[Solution.]
{\em If $(g,U)$ is admissible and $\deg(g,U)\neq 0$, then $g$ has a zero in $U$.}
\item[Additivity.]
{\em Let $(g,U)$ be admissible. If $U_1$ and $U_2$ are two
disjoint open subsets of $U$ whose union contains $g\sp{-1}(0) \cap U$, then}
\[
\deg(g,U) = \deg(g,U_1)+\deg(g,U_2).
\]
\item[Homotopy Invariance.]
{\em Let $h \colon U\times[0,1] \to \R\sp{k}$ be an admissible homotopy of
tangent vector fields; that is, $h(x,\lambda)\in T_xM$ for all $(x,\lambda)
\in U\times [0,1]$ and $h\sp{-1}(0)$ is compact. Then
$\deg\big(h(\cdot,\lambda),U\big)$ is independent of $\lambda$.}
\end{description}
As in the case of the fixed point index, the Additivity Property implies the following
important one:
\begin{description}
 \item[Excision]{\em Let $(g,U)$ be admissible. If $V\subseteq U$ is open and contains 
$g^{-1}(0)\cap U$, then $\deg(g,U) = \deg(g,V)$.} 
\end{description}

\section{Poincar\'e-type translation operators}

Let $M\subseteq\R^k$ be a manifold, and $g:M\to\R^k$ a tangent vector field on $M$. Let 
$f:\R\times M\times M\to\R^k$ be (continuous and) tangent to $M$ in the second variable 
(i.e.\ such that $f(t,p,q)\in T_pM$ for all $(t,p,q)\in\R\times M\times M$). Given $T>0$, 
assume also that $f$ is $T$-periodic in $t$.

Given $r>0$, consider the following delay differential equation depending on a parameter 
$\lambda\geq 0$:
\begin{equation}
\dot x(t)=g\big(x(t)\big)+\lambda f\big(t,x(t),x(t-r)\big).\label{ddern}
\end{equation}

We are interested in the $T$-periodic solutions of \eqref{ddern}. Without loss of 
generality we will assume that $T\geq r$ (\cite{Fr07}). In fact, for $n\in\N$, 
equations \eqref{ddern} and 
\[
\dot x(t)=g\big(x(t)\big)+\lambda f\big(t,x(t),x(t-(r-nT))\big)
\]
have the same $T$-periodic solutions. Thus, if necessary, one can replace $r$ with 
$r-nT$, where $n\in\N$ is such that $0<r-nT\leq T$. 

\smallskip
Let us introduce some notation. 

Given any $X\subseteq\R^k$, $\t X$ denotes the metric space $C\big([-r,0],X)$ with the 
distance inherited from the Banach space $\t\R^k = C([-r,0],\R^k)$ with the usual supremum 
norm. Notice that $\t X$ is complete if and only if $X$ is closed in $\R^k$.
Given any $p\in M$, denote by $\h p\in\t M$ the constant function $\h p(t)\equiv p$ 
and, for any $U\subseteq M$, define $\h U=\big\{\h p\in\t M:p\in U\big\}$. Also, given 
$W\subseteq\t M$, we put $\ck W=\big\{p\in M:\h p\in W\big\}$.

Observe that, for any given $U\subseteq M$, one has $\h U\subseteq\t U$ and
$\ck{\t U}=U$. It is known (see e.g.\ \cite{Ee66}) that $\t M$ is a smooth infinite
dimensional manifold. Actually, it turns out that it is a $C^1$--ANR (see e.g.\
\cite{EF76}), as a $C^1$ retract of the open subset $\t U$ of $\t\R^k$, $U$ being a 
tubular neighborhood of $M$ in $\R^k$. 

Assume now, till further notice, that $g$ is $C^1$. Consider the map $Q$ in $\t M$ 
defined by $Q(\varphi)(\theta)=x\big(\varphi(0),T+\theta\big)$, $\theta\in[-r,0]$, 
where $x(p,\cdot)$ denotes the unique solution of the Cauchy problem
\begin{subequations}\label{ddeg}
\begin{align}
&\dot x(t)=g\big(x(t)\big),\label{nopert}\\
&x(0)=p.
\end{align}
\end{subequations}
Well known properties of differential equations imply that the domain $\dom{Q}$ 
of $Q$ is an open subset of $\t M$. Also, since $T\geq r$, the Ascoli-Arzel\`a 
Theorem implies that $Q$ is a locally compact map (see, e.g.\ \cite{Ol69}).

Observe that the $T$-periodic solutions of \eqref{nopert} are in a one-to-one 
correspondence with the fixed points of $Q$. We will prove a formula (Theorem 
\ref{thgrado} below) for the computation of the fixed point index of the admissible
pairs $(Q,W)$, where $W$ is an open subset of $\dom{Q}$. Clearly, $Q$ is strictly 
related to the $M$-valued Poincar\'e map $P$, given by $P(p)=x(p,T)$, whose domain 
is the open subset $\dom{P}$ of $M$ consisting of those points $p$ such that the 
solution $x(p,\cdot)$ of the above Cauchy problem is defined up to $T$.

We shall need the following result of \cite{FS96} about the fixed point index of $P$.
\begin{theorem}\label{th2.1}
Let $g$ be as above and let $U\subseteq M$ be open and such that $\ind(P,U)$ is
defined. Then, $\deg(-g,U)$ is defined as well and
\[
 \ind(P,U)=\deg(-g,U).
\]
\end{theorem}

There is a simple relation between the domain $\dom{Q}$ of $Q$ and the domain 
$\dom{P}$ of $P$. In fact  $\dom{Q}=\{\varphi\in\t M:\varphi(0)\in\dom{P}\}$. 
In particular, $\t{\dom{P}}\subseteq\dom{Q}$. Observe also that 
$P(p)=Q(\h p)(0)$ for all $p\in\dom{P}$.

\begin{theorem}\label{thgrado}
Let $g$, $T$ and $Q$ be as above, and let $W\subseteq\t M$ be open. If the fixed 
point index $\ind(Q,W)$ is defined, then so is $\deg(-g,\ck W)$ and 
\[
\ind(Q,W)=\deg(-g,\ck W).
\]
\end{theorem}

\begin{proof}
The assumption that $\ind(Q,W)$ is defined means that $W\subseteq\dom{Q}$ and that 
$\fix (Q,W)$ is compact. Let us show that $\deg(-g,\ck W)$ is defined too. We need 
to prove that $g^{-1}(0)\cap\ck W$ is compact. If $p\in g^{-1}(0)\cap\ck W$, then 
the constant function $\h p$ is a fixed point of $Q$. Thus $g^{-1}(0)\cap\ck W$ is 
compact since it can be regarded as a closed subset of the compact set $\fix (Q,W)$.

We now use the Commutativity Property of the fixed point index in order to 
deduce the desired formula for the fixed point index of $Q$ from the analogous one 
for $P$, expressed in Theorem \ref{th2.1}. In order to do so, we define the maps 
$h:\dom{P}\to\t M$ and $k:\t M\to M$ by $h(p)(\theta)=x(p,\theta+T)$ and 
$k(\varphi)=\varphi(0)$, respectively. Clearly, we have
\begin{align}
&(h\circ k)(\varphi)(\theta) = x\big(\varphi(0),\theta+T\big)=Q(\varphi)(\theta),
&&\varphi\in\dom{Q},\quad \theta\in[-r,0],\label{compo1}\\
\intertext{and}
&(k\circ h)(p) = x(p,T)=P(p), && p\in\dom{P}.\label{compo2}
\end{align}

Define $\gamma=k|_W$. By the Commutativity Property of the fixed point index, 
$\ind\Big(h\circ\gamma,\gamma^{-1}\big(\dom{P}\big)\Big)$ is defined if and
only if so is $\ind\big(\gamma\circ h,h^{-1}(W)\big)$, and
\begin{equation}\label{indcomm}
\ind\Big(h\circ\gamma,\gamma^{-1}\big(\dom{P}\big)\Big)
                                 =\ind\big(\gamma\circ h,h^{-1}(W)\big).
\end{equation}

Since $W\subseteq\dom{Q}$, then $\gamma^{-1}\big(\dom{P}\big)=W$. Hence, from 
formulas \eqref{compo1}--\eqref{compo2}, it follows that
\begin{gather}
\ind(Q,W)=\ind\Big(h\circ\gamma,\gamma^{-1}\big(\dom{P}\big)\Big),\label{qcomp}\\
\ind\big(P,h^{-1}(W)\big)=\ind\big(\gamma\circ h,h^{-1}(W)\big).\label{pcomp}
\end{gather}
Thus, by \eqref{indcomm}, we get
\begin{equation}\label{indQPh}
\ind(Q,W)=\ind\big(P,h^{-1}(W)\big).
\end{equation}

From Theorem \ref{th2.1}, we obtain
\begin{equation}\label{vecchiaf}
\ind\big(P,h^{-1}(W)\big)=\deg\big(-g,h^{-1}(W)\big).
\end{equation}
From the definition of $h$ it follows immediately that
\begin{equation*}
g^{-1}(0)\cap\ck W=g^{-1}(0)\cap h^{-1}(W).
\end{equation*}
Therefore, from the Excision Property of the degree of a vector field, one has
\begin{equation}\label{idovvia}
\deg\big(-g,h^{-1}(W)\big)= \deg(-g,\ck W)
\end{equation}
and the assertion follows from equations \eqref{indQPh}, \eqref{vecchiaf} and 
\eqref{idovvia}.
\end{proof}

Let $W\subseteq\dom{Q}$ be open in $\t M$. We point out that Theorems \ref{th2.1} 
and \ref{thgrado} imply that the fixed point index of $Q$ in $W$ actually reduces 
to the fixed point index of the finite dimensional operator $P$ in $\ck W$. Namely,
\begin{equation}\label{frdf}
 \ind(Q,W)=\ind(P,\ck W).
\end{equation}
In fact, $P$ is defined on $\ck{W}$ and $\fix(P,\ck{W})$ can be regarded as a 
closed subset of $\fix(Q,W)$. Therefore, if  $\ind(Q,W)$ is defined, then so is 
$\ind(P,\ck W)$ and, applying Theorems \ref{thgrado} and \ref{th2.1}, we get 
$\ind(Q,W)=\deg(-g,\ck W)=\ind(P,\ck W)$.

Let us remark that the mappings $h$ and $k$, defined in the proof of Theorem \ref{thgrado},
establish a bijection between the fixed point sets of $Q$ and $P$. However, we should 
not think of formula \eqref{frdf} as a trivial consequence of this correspondence. In
fact, given $W$ as in Theorem \ref{thgrado}, we see that $h$ and $k$ induce a one-to-one
correspondence between the fixed points of $Q$ in $W$ and those of $P$ in $h^{-1}(W)$ but, 
in general, $\fix\big(P,h^{-1}(W)\big)\neq\fix\big(P,\ck W\big)$. Observe also that the 
``finite dimensional reduction formula'' \eqref{frdf} has a clear advantage over the more 
crude reduction formula \eqref{indQPh} obtained in the proof of Theorem \ref{thgrado} by 
means of the Commutativity Property of the fixed point index (and that derives from the 
correspondence we just mentioned). In fact, differently from the set $h^{-1}(W)$ that 
appears in \eqref{indQPh}, the open set $\ck W$ does not depend on the equation
\eqref{nopert}.

\section{Branches of starting pairs}

Any pair $(\lambda,\varphi)\in [0,\infty)\times\t M$ is said to be a \emph{starting 
pair} (for \eqref{ddern}) if the following initial value problem has a $T$-periodic 
solution:
\begin{equation}\label{due.uno}
\left\{
\begin{array}{ll}
\dot x(t)=g\big(x(t)\big)+\lambda f\big(t,x(t),x(t-r)\big) &  t>0,\\
x(t)=\varphi(t), & t\in [-r,0].
\end{array}
\right.
\end{equation}
A pair of the type $(0,\h p)$ with $g(p)=0$ is clearly a starting pair and will be called 
a \emph{trivial starting pair}. The set of all starting pairs for \eqref{ddern} will be 
denoted by $S$.

Throughout this section we shall assume that $f$ and $g$ are $C^1$, so that \eqref{due.uno} 
admits a unique solution that we shall denote by $\xi^\lambda(\varphi,\cdot)$. Observe that
$\xi^0\big(\varphi(0),\cdot\big)=x\big(\varphi(0),\cdot\big)$, where, we recall, $x(p,\cdot)$
is the unique solutions of the Cauchy problem \eqref{ddeg}. By known continuous dependence 
properties of delay differential equations the set $\V\subseteq[0,\infty)\times\t M$ given by
\begin{equation*}
\V:=\big\{(\lambda,\varphi): \xi^\lambda(\varphi,\cdot)
                        \textrm{ is defined on $[0,T]$}\big\}
\end{equation*}
is open. Clearly $\V$ contains the set $S$ of all starting pairs for \eqref{ddern}. 
Observe that $S$ is closed in $\V$, even if it could be not so in $[0,+\infty)\times\t M$. 
Moreover, by the Ascoli-Arzel\`a Theorem it follows that $S$ is locally compact.  

In the sequel, given $A\subseteq\R\times\t M$ and $\lambda\in\R$, we will denote the slice 
$\{x\in\t M:(\lambda,x)\in A\}$ by the symbol $A_\lambda$. Observe that $\ck \V_0=\dom{P}$
where $P$ is the Poincar\'e operator defined in the previous section.

In order to study the $T$-periodic solutions of \eqref{eqintro}, it will be convenient to
introduce, for each $\lambda\geq 0$, the map $Q_\lambda:\V_\lambda\to\t M$ given by
\[
 Q_\lambda(\varphi)(\theta)=\xi^\lambda(\varphi,\theta+T),\quad \theta\in[-r,0].
\]
Notice that $Q_0$ coincides with the map $Q$ defined in the previous section.

We will need the following global connectivity result of \cite{FP93}.
\begin{lemma}\label{L.3.1}
Let $Y$ be a locally compact metric space and let $Z$ be a compact subset of $Y$. 
Assume that any compact subset of $Y$ containing $Y_0$ has nonempty boundary. Then 
$Y\setminus Z$ contains a connected set whose closure (in $Y$) intersects 
$Z$ and is not compact.
\end{lemma}

\begin{proposition}\label{T.3.1}
Assume that $f$, $g$, $S$ are as above. Given $W\subseteq[0,\infty]\times\t M$ open, 
if $\deg(g,\ck W_0)$ is (defined and) nonzero, then the set 
\[
(S\cap W)\setminus \big\{(0,\h p)\in W:g(p)=0\big\}
\]
of nontrivial starting pairs in $W$, admits a connected subset whose closure in $S\cap W$ 
meets $\big\{(0,\h p)\in W:g(p)=0\big\}$ and is not compact.
\end{proposition}

\begin{proof}
Let us define the open set $U=W\cap \V$. Since $g^{-1}(0)\cap\ck U_0=g^{-1}(0)\cap\ck W_0$, 
and $S\cap U=S\cap W$, we need to prove that the set of nontrivial starting pairs in $U$ 
admits a connected subset whose closure in $S\cap U$ meets 
$\big\{(0,\h p)\in U:g(p)=0\big\}$ and is not compact.

As pointed out before, $S$ is locally compact, thus, $U$ being open, $S\cap U$ is locally 
compact. Moreover the assumption that $\deg(g,\ck W_0)$ is defined means that the set
\[
 \big\{p\in\ck{W}_0:g(p)=0\big\}=\big\{p\in\ck{U}_0:g(p)=0\big\}
\]
is compact. Thus the homeomorphic set $\{(0,\h p)\in U:g(p)=0\}$ is compact as well.

The assertion will follow applying Lemma \ref{L.3.1} to the pair
\[
(Y,Z)=\left(S\cap U ,\big\{(0,\h p)\in U:p\in g^{-1}(0)\big\}\right).
\]
In fact, if $\Sigma$ is a connected set as in the assertion of Lemma \ref{L.3.1}, its 
closure satisfies the requirement. 

Assume, by contradiction, that there exists a compact subset $C$ of the set $S\cap U$ of 
starting pairs of \eqref{due.uno} in $U$ containing $Z$ and with empty boundary in 
$S\cap U$. Thus $C$ is a relatively open subset of $S\cap U$. As a consequence, 
$(S\cap U)\setminus C$ is closed in $S\cap U$, so the distance, 
$\delta=\dist\big(C,(S\cap U)\setminus C\big)$, between $C$ and $(S\cap U)\setminus C$ is 
nonzero (recall that $C$ is compact). Consider the set
\[
A=\left\{(\lambda,\varphi)\in U: 
   \dist\big((\lambda,\varphi),C\big)<\delta/2\right\},
\]
which, clearly, does not meet $(S\cap U)\setminus C$. 

Because of the compactness of $S\cap U\cap A=C$, there exists $\cl\lambda>0$ such that 
$(\{\cl\lambda\}\times A_{\cl\lambda})\cap S\cap U=\emptyset$. 
Moreover, the set $S\cap U\cap A$ coincides with 
$\{(\lambda,\varphi)\in A:Q_\lambda(\varphi)=\varphi\}$. 
Then, from the Generalized Homotopy Invariance Property of the fixed point index,
\[
0=\ind\big(Q_{\cl\lambda},A_{\cl\lambda}\big)=\ind\big(Q_\lambda,A_\lambda\big),
\]
for all $\lambda\in [0,\cl\lambda]$. But, by Theorem \ref{thgrado} and by the Excision 
Property of the degree, we get
\[
\ind(Q,A_0)=\deg(-g,\ck A_0)=\deg (-g,\ck W_0)\neq 0.
\]
That contradicts the previous formula, since $Q=Q_0$.
\end{proof}

\section{Branches of $T$-periodic pairs}

Let us introduce the function space where most of the work of this section is done.
We will denote by $C_T(M)$ the metric subspace of the Banach space 
$\left( C_T(\R^k)\,,\,\left\| \cdot \right\|\right)$ of all the $T$-periodic 
continuous maps $x:\R\to M$ with the usual $C^0$ norm. Observe that $C_T(M)$ is not 
complete unless $M$ is complete (i.e.\ closed in $\R^k$). Nevertheless, since $M$ is 
locally compact, $C_T(M)$ is always locally complete. 
\smallskip\ 

For the sake of simplicity, we will identify $M$ with its image in 
$[0,\infty)\times C_T(M)$ under the embedding which associates to any $p\in M$ the 
pair $(0,\bar p)$, $\bar p\in C_T(M)$ being the map constantly equal to $p$. 
According to these identifications, if $E$ is a subset of $[0,\infty)\times C_T(M)$, 
by $E\cap M$ we mean the subset of $M$ given by all $p\in M$ such that the pair 
$(0,\bar p)$ belongs to $E$. Observe that if $\Omega\subseteq[0,\infty)\times C_T(M)$ 
is open, then so is $\Omega\cap M$.

A pair $(\lambda,x)\in [0,\infty)\times C_T(M)$, where $x$ a solution of \eqref{ddern},
is called a \emph{$T$-periodic pair} (for \eqref{ddern}). Those $T$-periodic pairs that 
are of the particular form $(0,\cl{p})$ are said to be \emph{trivial}. Observe that
$(0,\cl p)\in [0,\infty)\times C_T(M)$ is a trivial $T$-periodic pair if and only if 
$g(p)=0$. We point out that if $x$ is a nonconstant $T$-periodic solution of the
unperturbed equation $\dot x(t)=g\big(x(t)\big)$, then $(0,x)$ is a nontrivial 
$T$-periodic pair.

We are now in a position to state our main result. The proof is inspired by 
\cite{FS96,FP93}.

\begin{theorem}\label{tuno}
Let $g:M\to\R^k$ be a tangent vector field on $M$ and, given $T>0$, let 
$f:\R\times M\times M\to\R^k$ be $T$-periodic in the first variable and tangent to $M$ 
in the second one. Let $\Omega$ be an open subset of $[0,\infty)\times C_T(M)$, and 
assume that $\deg(g,\Omega\cap M)$ is defined and nonzero. Then $\Omega$ contains a
connected set of nontrivial $T$-periodic pairs whose closure in $\Omega$ meets the 
set $\{(0,\cl p)\in\Omega:g(p)=0\}$ and is not compact. 

In particular, the set of $T$-periodic pairs for \eqref{ddern} contains a connected
component that meets $\{(0,\cl p)\in\Omega:g(p)=0\}$ and whose intersection with $\Omega$
is not compact.
\end{theorem}

\begin{proof} 
Denote by $X$ the set of $T$-periodic pairs of \eqref{ddern} and by $S$ the set of 
starting pairs of the same equation; that is, of all pairs $\big(\lambda,x|_{[-r,0]}\big)$ 
with $(\lambda,x)\in X$, $x|_{[-r,0]}$ being the restriction to $[-r,0]$ of $x$.

Assume first that $f$ and $g$ are smooth. Define the map $h:X\to S$ by $h(\lambda,x)
=\big(\lambda,x|_{[-r,0]}\big)$ and observe that $h$ is continuous, onto and, since $f$ and 
$g$ are smooth, it is also one to one. Furthermore, by the continuous dependence on data, 
$h^{-1}:S\to X$ is continuous as well.

Take 
\[
S_\Omega=\left\{ (\lambda ,\varphi)\in S:\mbox{the solution of \eqref{ddern} is
contained in }\Omega \right\} . 
\]
So that $X\cap\Omega$ and $S_\Omega$ correspond under the homeomorphism $h:X\to S$. Thus, 
$S_\Omega$ is an open subset of $S$ and, consequently, we can find an open subset $W$ 
of $[0,\infty)\times\t M$ such that $S\cap W=S_\Omega$. This implies
\begin{multline*}
\big\{ p\in\ck W_0:g(p)=0\big\}=\big\{p\in M:(0,\h p)\in W,\; g(p)=0\big\}=\\
=\big\{p\in M:(0,\cl{p})\in\Omega,\; g(p)=0\big\}=\big\{p\in\Omega\cap M:g(p)=0\big\}.
\end{multline*}
Thus, by excision, $\deg(g,\ck W_0)=\deg(g,\Omega\cap M)\neq 0$. Applying Proposition 
\ref{T.3.1}, we get the existence of a connected set 
\[
\Sigma\subseteq ( S\cap W)\setminus\big\{(0,\h p)\in W:g(p)=0\big\} 
\] 
whose closure in $S\cap W$ meets $\big\{(0,\h p)\in W:g(p)=0\big\}$ and is not compact. 

Observe that the trivial $T$-periodic pairs correspond to the trivial starting 
pairs under the homeomorphism $h$. 
Thus, $\Gamma = h^{-1}(\Sigma)\subseteq X\cap\Omega$ is a connected set of nontrivial 
$T$-periodic pairs whose closure in $X\cap\Omega$ meets $\{(0,\cl p)\in\Omega:g(p)=0\}$
and is not compact. Since $X$ is closed in $[0,\infty)\times C_T(M)$, the closures of
$\Gamma$ in $X\cap\Omega$ and in $\Omega$ coincide. This proves that $\Gamma$
satisfies the requirements of the first part of the assertion.\smallskip\ 

Let us remove the smoothness assumption on $g$ and $f$. As above, it is enough to show 
the existence of a connected set $\Gamma$ of nontrivial $T$-periodic pairs whose closure 
in $X\cap\Omega$ meets $\{(0,\cl p)\in\Omega:g(p)=0\}$ and is not compact.

Observe that the closed subset $X$ of $[0,\infty)\times C_T(M)$ is locally compact because 
of Ascoli-Arzel\`a Theorem. It is convenient to introduce the following subset of $X$:
\[
\Upsilon=\big\{(0,\cl{p})\in [0,\infty)\times C_T(M):g(p)=0\big\}.
\]
Take 
\[
Y=X\cap\Omega\quad\text{and}\quad Z=\Upsilon\cap\Omega
\] 
and notice that $Y$ is locally compact as an open subset of $X$. Moreover, $Z$ is a compact 
subset of $Y$ (recall that, by assumption, $\deg(g,M\cap\Omega)$ is defined). Since $Y$ is 
closed in $\Omega$, we only have to prove that the pair $(Y,Z)$ satisfies the hypothesis of 
Lemma \ref{L.3.1}. Assume the contrary. Thus, we can find a relatively open compact subset 
$C$ of $Y$ containing $Z$. Similarly to the proof of Proposition \ref{T.3.1}, given 
$0<\rho<\dist (C,Y\setminus C)$, we consider the set $A^\rho$ of all pairs 
$(\lambda,\varphi)\in\Omega$ whose distance from $C$ is smaller than $\rho$. Thus, 
$A^\rho\cap Y=C$ and $\partial A^\rho\cap Y=\emptyset$. We can also assume that the closure 
$\cl{A^\rho}$ of $A^\rho$ in $[0,\infty)\times C_T(M)$ is contained in $\Omega$. Since $C$ is 
compact and $[0,\infty)\times M$ is locally compact, we can take $A^\rho$ in such a way that 
the set 
\[
\left\{ \big(\lambda \,,\,x(t),x(t-r)\big) \in [ 0,\infty)\times M\times M:
               (\lambda ,x)\in A^\rho,\; t\in [0,T]\right\} 
\]
is contained in a compact subset of $[0,\infty)\times M\times M$. This implies that $A^\rho$ is 
bounded with complete closure and $A^\rho\cap M$ is a relatively compact subset of $\Omega\cap M$. 
In particular $g$ is nonzero on the boundary of $A^\rho\cap M$ (relative to $M$). By well known 
approximation results on manifolds, we can find sequences $\{g_i\} $ and $\{ f_i\} $ of smooth 
maps uniformly approximating $g$ and $f$, and such that the following properties hold for all 
$i\in\N$:
\begin{list}{$\bullet$}{\setlength{\itemsep}{3pt}%
                         \setlength{\topsep}{3pt}%
                         \setlength{\partopsep}{0pt}%
                         \setlength{\parsep}{0pt}%
                         \setlength{\parskip}{0pt}}
 \item  $g_i(p)\in T_pM$ for all $p\in M$;
 \item  $f_i(t,p,q)\in T_pM$ for all $(t,p,q)\in\R\times M\times M$;
 \item  $f_i$ is $T$-periodic in the first variable.
\end{list}

\medskip
For $i\in\N$ large enough, we get 
\[
\deg (g_i,A^\rho\cap M)=\deg (g,A^\rho\cap M). 
\]
Furthermore, by excision, 
\[
\deg (g,A^\rho\cap M)=\deg (g,\Omega \cap M)\neq 0. 
\]
Therefore, given $i$ large enough, the first part of the proof can be applied to 
the equation 
\begin{equation}
\dot x(t)=g_i\big(x(t)\big)+\lambda f_i\big(t,x(t),x(t-r)\big).  \label{due}
\end{equation}
Let $X_i$ denote the set of $T$-periodic pairs of \eqref{due} and put
\[
\Upsilon_i=\big\{(0,\cl{p})\in [0,\infty)\times C_T(M):g_i(p)=0\big\}.
\]
Because of the first part of the proof, there exists a connected subset $\Gamma_i$ of 
$A^\rho$ whose closure in $A^\rho$ meets $\Upsilon_i\cap A^\rho$ and is not compact. Let 
us denote by $\cl\Gamma_i$ and $\cl{A^\rho}$ the closures in $[0,\infty)\times C_T(M)$ of 
$\Gamma_i$ and $A^\rho$, respectively. 

Let us show that, for $i$ large enough, $\cl\Gamma_i\cap\partial A^\rho\neq\emptyset$. 
Thus, $X_i$ being closed, we get $\cl{\Gamma}_i\subseteq X_i$. This will imply the existence
of a $T$-periodic pair $(\lambda_i,x_i)\in\partial A^\rho$ of \eqref{due}. It is 
enough to prove that $\cl\Gamma_i$ is compact. In fact, if this is true and if we assume 
$\cl\Gamma_i\cap\partial A^\rho=\emptyset$, then $\cl\Gamma_i\subseteq A^\rho$ which implies that 
the closure of $\Gamma_i$ in $A^\rho$ coincides with the compact set $\cl\Gamma_i$, and this 
is a contradiction. The compactness of $\cl\Gamma_i$, for $i$ large enough, follows from 
the completeness of $\cl{A^\rho}$ and the fact that, by the Ascoli-Arzel\`a Theorem, 
$\cl\Gamma_i$ is totally bounded, when $i$ is sufficiently large. Thus, for $i$ large 
enough, there exists a $T$-periodic pair $(\lambda_i,x_i)\in\partial A^\rho$ of \eqref{due}. 

Again by Ascoli-Arzel\`a Theorem, we may assume that $x_i\to x_0$ in $C_T(M)$ and 
$\lambda_i\to\lambda_0$ with $(\lambda_0,x_0)\in\partial A^\rho$. Passing to the 
limit in equation \eqref{due}, it is not difficult to show that $(\lambda_0,x_0)$ 
is a $T$-periodic pair of \eqref{ddern} in $\partial A^\rho$. This contradicts the assumption 
$\partial A^\rho\cap Y=\emptyset$ and proves the first part of the assertion.

Let us prove the last part of the assertion. Consider the connected component $\Xi$ 
of $X$ that contains the connected set $\Gamma$ of the first part of the assertion.
We shall now show that $\Xi$ has the required properties. Clearly, $\Xi$ meets the set
$\big\{(0,\cl{p})\in\Omega:g(p)=0\big\}$ because the closure $\cl{\Gamma}^\Omega$ of 
$\Gamma$ in $\Omega$ does. Moreover, $\Xi\cap\Omega$ cannot be compact, since 
$\Xi\cap\Omega$, as a closed subset of $\Omega$, contains $\cl{\Gamma}^\Omega$, and $\cl{\Gamma}^\Omega$ is not compact.
\end{proof}

The following corollary, in the case of a compact boundaryless manifolds, extends 
a result of \cite{BCFP07} in which $g$ is identically zero.

\begin{corollary}
 Let $f$ and $g$ be as in Theorem \ref{tuno} and let $M\subseteq\R^k$ be compact with 
nonzero Euler-Poincar\'e characteristic $\chi(M)$. Then, there exists an unbounded connected 
set of nontrivial $T$-periodic pairs whose closure meets 
$\{(0,\cl p)\in[0,\infty)\times C_T(M):g(p)=0\}$. 
In particular, equation \eqref{due.uno} has a solution for any $\lambda\geq 0$.
\end{corollary}

\begin{proof}
Since $M$ is compact, $[0,\infty)\times C_T(M)$ is a complete metric space.
Moreover, the Ascoli-Arzel\`a Theorem implies that any bounded set of $T$-periodic
pairs is totally bounded. The Poincar\'e-Hopf Theorem yields $\deg(g,M)=\chi(M)\neq 0$. 
Thus, Theorem \ref{tuno} implies the existence of an unbounded connected set $\Gamma$ of 
nontrivial $T$-periodic pairs whose closure in $[0,\infty)\times C_T(M)$ meets 
$\{(0,\cl p)\in[0,\infty)\times C_T(M):g(p)=0\}$. The last assertion follows from
the fact that $C_T(M)$ is bounded while $\Gamma$ is unbounded.
\end{proof}

\begin{corollary}\label{cor1}
Let $f$ and $g$ be as in Theorem \ref{tuno}. Assume that $M$ is closed as a subset 
of $\R^k$. Let $\Omega\subseteq [0,\infty)\times C_T(M)$ be open and such that 
$\deg(g,\Omega\cap M)$ is defined and nonzero. Then there exists a connected component 
$\Gamma$ of $T$-periodic pairs that meets $\big\{(0,\cl p)\in\Omega:g(p)=0\big\}$ and 
cannot be both bounded and contained in $\Omega$. In particular, if $\Omega$ is 
bounded, then $\Gamma\cap \partial\Omega\neq\emptyset$.
\end{corollary}

\begin{proof}
Since $M$ is a closed subset of $\R^k$, $[0,\infty)\times C_T(M)$ is complete.
Moreover, the Ascoli-Arzel\`a Theorem implies that any bounded set of $T$-periodic
pairs is totally bounded. Thus, the first part of the assertion follows from Theorem 
\ref{tuno}. The last part of the assertion follows from the fact that $\Gamma$ is 
connected and that $\emptyset\neq\Gamma\cap\Omega\neq\Gamma$.
\end{proof}

\medskip
To better understand the meaning of Corollary \ref{cor1}, consider for example the 
case when $M=\R^m$. If $g^{-1}(0)$ is compact and $\deg (g,\R^m)\neq 0$, then there 
exists an unbounded connected set of $T$-periodic pairs in $[0,\infty)\times C_T(\R^m)$ 
which meets the set $\{(0,\cl p)\in[0,\infty)\times C_T(M):g(p)=0\}$, that can be
identified with $g^{-1}(0)$. The existence of this unbounded connected set cannot be 
destroyed by a particular choice of $f$. However it is possibly ``completely vertical'',
i.e.\ contained in the slice $\{0\}\times C_T(M)$. This peculiarity is exhibited,
for instance, by the set of $T$-periodic pairs of the equation
\begin{equation*}
\left\{
\begin{array}{l}
\dot x=y,\\
\dot y=-x +\lambda \sin t,
\end{array}\right.
\end{equation*}
where $M=\R^2$ and $T=2\pi$.

A somewhat opposite behavior is shown by the set $X$ of $T$-periodic pairs for \eqref{ddern} 
in the ``degenerate'' situation when $f(t,p,q)\equiv0$. In this case, $X$ consists 
of the pairs $(\lambda,x)$, where $\lambda\geq 0$ and $x$ is a $T$-periodic solution to 
$\dot x=g(x)$.
In particular, given any $p \in M$ such that $g(p)=0$, the connected component $\Gamma$ of $X$ containing $\{0\} \times \cl p$  contains the ``horizontal'' set $[0,+\infty) \times \{\cl p\}$ and, clearly, satisfies the requirement of Corollary \ref{cor1}.

\end{document}